\documentclass[12pt]{amsart}
\usepackage{amssymb, latexsym}
\usepackage{hyperref}

\DeclareMathOperator{\sgn}{\mathrm{sgn}}

\begin{document}
 \bibliographystyle{plain}

 \newtheorem{theorem}{Theorem}
 \newtheorem{lemma}{Lemma}
 \newtheorem{corollary}{Corollary}
 \newcommand{\mc}{\mathcal}
 \newcommand{\rar}{\rightarrow}
 \newcommand{\Rar}{\Rightarrow}
 \newcommand{\lar}{\leftarrow}
 \newcommand{\mbb}{\mathbb}
 \newcommand{\A}{\mc{A}}
 \newcommand{\B}{\mc{B}}
 \newcommand{\cc}{\mc{C}}
 \newcommand{\D}{\mc{D}}
 \newcommand{\E}{\mc{E}}
 \newcommand{\F}{\mc{F}}
 \newcommand{\G}{\mc{G}}
 \newcommand{\HH}{\mc{H}}
 \newcommand{\I}{\mc{I}}
 \newcommand{\J}{\mc{J}}
 \newcommand{\M}{\mc{M}}
 \newcommand{\nn}{\mc{N}}
 \newcommand{\qq}{\mc{Q}}
 \newcommand{\U}{\mc{U}}
 \newcommand{\X}{\mc{X}}
 \newcommand{\Y}{\mc{Y}}
 \newcommand{\C}{\mathbb{C}}
 \newcommand{\R}{\mathbb{R}}
 \newcommand{\N}{\mathbb{N}}
 \newcommand{\Q}{\mathbb{Q}}
 \newcommand{\Z}{\mathbb{Z}}
 \newcommand{\lf}{\left\lfloor}
 \newcommand{\rf}{\right\rfloor}
 \newcommand{\dx}{\text{\rm d}x}
 \newcommand{\dy}{\text{\rm d}y}
 \newcommand{\SL}{\mathrm{SL}}

\parskip=0.5ex

\title[Inhomogeneous approximation by coprime integers]{Inhomogeneous approximation\\ by coprime integers}
\author{Alan~Haynes}
\thanks{Research supported by EPSRC grant EP/J00149X/1.}
\address{School of Mathematics, University of Bristol, Bristol UK }
\email{alan.haynes@bristol.ac.uk}

\allowdisplaybreaks


\begin{abstract}
This paper addresses a problem recently raised by Laurent and Nogueira about inhomogeneous Diophantine approximation with coprime integers. A corollary of our main theorem is that for any irrational $\alpha\in\R$ and for any $\gamma\in\R$ and $\epsilon>0$ there are infinitely many pairs of coprime integers $m,n$ such that
\begin{equation*}
|n\alpha-m-\gamma |\le 1/|n|^{1-\epsilon}.
\end{equation*}
This improves upon previously known results, in which the exponent of approximation was at best $1/2$.
\end{abstract}


\maketitle

\section{Introduction}
Dirichlet's theorem in Diophantine approximation guarantees that for any irrational $\alpha\in\R$ there are infinitely many $m,n\in\Z$ for which
\begin{equation*}
\left|n\alpha-m\right|\le\frac{1}{|n|}.
\end{equation*}
The inhomogeneous version of this result, proved by Minkowski (see \cite[Theorem IV.9.1]{RockettSzusz}), is that for any irrational $\alpha\in\R$ and for any $\gamma\in\R\setminus (\alpha\Z+\Z)$ there are infinitely many $m,n\in\Z$ for which
\begin{equation*}
\left|n\alpha-m-\gamma\right|\le\frac{1}{4|n|}.
\end{equation*}
In this paper we address the problem of obtaining analogous results with $m$ and $n$ coprime. In the homogeneous case there is little need to pause for thought, since any common factors can be dispensed with immediately without significantly changing the problem. However in the inhomogeneous case the situation is more delicate. The best known analogue of Dirichlet's theorem in this general setting is a recent result of Laurent and Nogueira, who proved in \cite{LaurentNoguiera} that for any irrational $\alpha\in\R$ and for any $\gamma\in\R$, there are infinitely many pairs of coprime integers $m$ and $n$ such that
\[\left|n\alpha-m-\gamma\right|\le\frac{c}{|n|^{1/2}},\]
where $c$ is a constant depending only on $\alpha$ and $\gamma$. Their proof relies on estimates for the density of orbits of points in $\R^2$ under the action of $\SL_2(\Z)$. In this paper, using a different approach, we obtain the following result.
\begin{theorem}\label{thm:main}
Let $c>2\sqrt{\log 2}$. For any irrational $\alpha\in\R$ and for any $\gamma\in\R$ there are infinitely many pairs of coprime integers $m,n$ such that
\begin{equation}\label{eqn:mainthm}
\left|n\alpha-m-\gamma\right|\le\frac{\exp (c\sqrt{\log |n|})}{|n|}.
\end{equation}
\end{theorem}
As $n\rar\infty$ the function $\exp (c\sqrt{\log n})$ grows asymptotically more slowly than any power of $n$, and so we have the following immediate corollary.
\begin{corollary}
For any irrational $\alpha\in\R$ and for any $\gamma\in\R$ and $\epsilon>0$ there are infinitely many pairs of coprime integers $m,n$ such that
\begin{equation*}
\left|n\alpha-m-\gamma\right|\le\frac{1}{|n|^{1-\epsilon}}.
\end{equation*}
\end{corollary}
Our method uses only elementary techniques and it seems plausible that by a refinement one might be able to replace the right hand side of (\ref{eqn:mainthm}) by $c'/|n|$. This would clearly be best possible, apart from the determination of the best constant $c'$.

\section{Preliminary results}\label{sec:prelim}
\subsection{Continued fractions}
We write the simple continued fraction expansion of an irrational real number $\alpha$ as
\begin{align*}
\alpha = a_0 + \cfrac{1}{a_1+
            \cfrac{1}{a_2+
             \cfrac{1}{a_3+\dotsb}}}=[a_0; a_1, a_2, a_3, \dots ],
\end{align*}
where $a_0$ is an integer and $a_1, a_2, \dots $ is a sequence of positive integers uniquely determined by
$\alpha$. The rational numbers
\[\frac{p_k}{q_k}=[a_0;a_1,\ldots ,a_k],\quad k\ge 0,\]
are the principal convergents to $\alpha$, and it is assumed that $p_k$ and $q_k$ are coprime and that $q_k>0$ for all $k$. For $k\ge 0$ we also write
\begin{equation*}
D_k=q_k\alpha-p_k.
\end{equation*}
We have by the basic properties of continued fractions that for $k\ge 1$,
\begin{equation}\label{eqn:cfprop1}
p_{k+1}=a_{k+1}p_k+p_{k-1},\qquad q_{k+1}=a_{k+1}q_k+q_{k-1},
\end{equation}
\begin{equation}\label{eqn:cfprop2}
p_kq_{k-1}-q_kp_{k-1}=(-1)^{k+1},\quad\text{ and}
\end{equation}
\begin{equation}\label{eqn:D_kineq}
(-1)^kD_k=\left|q_k\alpha-p_k\right|\le\frac{1}{q_{k+1}}.
\end{equation}
For fixed irrational $\alpha$ we can use the greedy algorithm to represent any natural number uniquely as a weighted sum of the $q_k$'s, where the $k$th weight does not exceed $a_{k+1}$. This is made precise by the following lemma.
\begin{lemma}\label{lem:intcfexp}\cite[Section II.4]{RockettSzusz}
Suppose $\alpha\in\R$ is irrational. Then for every $n\in\N$ there is a unique integer $M\ge0$ and a unique sequence $\{c_{k+1}\}_{k=0}^\infty$ of integers such that $q_M\le n< q_{M+1}$ and
\begin{equation}\label{Ostexp1}
n=\sum_{k=0}^\infty c_{k+1}q_k,
\end{equation}
\begin{equation*}
\text{with }~0\le c_1<a_1,\quad 0\le c_{k+1}\le a_{k+1}\ \text{ for } \ k\ge 1,
\end{equation*}
\begin{equation*}
c_k=0 \quad \text{whenever}\quad c_{k+1}=a_{k+1}\ \text{ for some }k\ge 1, \qquad\text{and}
\end{equation*}
\begin{equation*}
c_{k+1}=0 \quad \text{for}  \quad k>M.
\end{equation*}
\end{lemma}
Furthermore we can construct a similar expansion for real numbers by using the $D_k$'s in place of the $q_k$'s. In what follows $\{x\}$ denotes the fractional part of a real number $x$.
\begin{lemma}\label{lem:realcfexp}\cite[Theorem II.6.1]{RockettSzusz}
Suppose $\alpha\in\R\setminus\Q$ has continued fraction expansion as above. For any $\gamma\in [-\{\alpha\},1-\{\alpha\})\setminus (\alpha\Z+\Z)$ there is a unique sequence $\{b_{k+1}\}_{k=0}^{\infty}$ of integers such that
\begin{equation}\label{Ostexp2}
\gamma=\sum_{k=0}^\infty b_{k+1}D_k,
\end{equation}
\begin{equation*}
\text{with }~0\le b_1<a_1,\quad 0\le b_{k+1}\le a_{k+1}\ \text{ for } \ k\ge 1,\qquad\text{and}
\end{equation*}
\begin{equation*}
b_k=0 \quad \text{whenever}\quad b_{k+1}=a_{k+1}\ \text{ for some }k\ge 1.
\end{equation*}
\end{lemma}
The relevance of these expansions to inhomogeneous approximation is explained by the following result, which can be deduced from the arguments in \cite[Section II.6]{RockettSzusz} (a rigorous proof can be found in \cite{BHVV}).
\begin{lemma}\label{lem:inhomappbnd}
Let $\alpha\in\R\setminus\Q$ and suppose that $\gamma\in [-\{\alpha\},1-\{\alpha\})\setminus(\alpha\Z+\Z).$ Choose an integer $n\in\N$ and, referring to the expansions (\ref{Ostexp1}) and (\ref{Ostexp2}), write $\delta_{k+1}=c_{k+1}-b_{k+1}$ for $k\ge 0$. If $\delta_{k+1}=0$ for all $k<m$ and some $m\ge 4$ then
\begin{equation*}
\left|n\alpha-\sum_{k=0}^Mc_{k+1}p_k-\gamma\right|\le \frac{3\max(1,|\delta_{m+1}|)}{q_{m+1}}.
\end{equation*}
\end{lemma}
Finally we will use the well known fact that
\begin{align}\label{eqn:cfprop3}
\frac{p_{k-1}}{q_{k-1}}<\alpha<\frac{p_k}{q_k}~\text{ for } k \text{ odd},
\end{align}
and we will also need the following inhomogeneous version of this fact.
\begin{lemma}\label{lem:D_ksumsgn}
Suppose $\alpha\in\R\setminus\Q$ and $\gamma\in [-\{\alpha\}, 1-\{\alpha\}]\setminus (\alpha\Z+\Z)$. Then, using the notation of Lemma \ref{lem:realcfexp}, if $m\ge 4$ and $b_{m+1}\not= 0$ we have that
\begin{equation}\label{eqn:D_ksumsgn1}
\sgn\left(\sum_{k=m}^\infty b_{k+1}D_k\right)=(-1)^m.
\end{equation}
\end{lemma}
\begin{proof}
We have from (\ref{eqn:D_kineq}) that $\sgn (D_k)=(-1)^k$ for $k\ge 1$. Therefore the terms in (\ref{eqn:D_ksumsgn1}) with opposite sign to $D_m$, when added together, are no larger in absolute value than
\begin{align*}
\left|a_{m+2}D_{m+1}+a_{m+4}D_{m+3}+\cdots\right|.
\end{align*}
By (\ref{eqn:cfprop1}) this expression is equal to
\begin{align*}
\left|(D_{m+2}-D_{m})+(D_{m+4}-D_{m+2})+(D_{m+6}-D_{m+4})+\cdots\right|=|D_m|,
\end{align*}
and this shows that
\[(-1)^m\left(\sum_{k=m}^\infty b_{k+1}D_k\right)\ge 0.\]
Furthermore the assumption that $\gamma\not\in (\alpha\Z+\Z)$ means that there cannot be equality in this inequality, so we are finished.
\end{proof}

\subsection{Estimates from elementary number theory}
In what follows $\mu$ denotes the M\"{o}bius function, $\varphi$ the Euler-phi function, $\omega (n)$ the number of distinct prime factors of $n$, $\pi(x)$ the number of primes $\le x$, and $(m,n)$ the greatest common divisor of $m$ and $n$. The letter $p$, without a subscript, will always denote a prime number (not to be confused with the quantities $p_k$ coming from continued fractions). We will use the Landau and Vinogradov asymptotic notation with the standard meaning for the symbols $\ll, \gg, O(\cdot), o(\cdot),$ and $\sim$, and all implied constants will be universal unless otherwise indicated. All summations are restricted to positive integers.

For use in what follows we remind the reader of two well known results of Mertens (see \cite[Theorems 427, 428]{HardyWright}), that
\begin{align}
&\sum_{p\le x}\frac{1}{p}\sim \log\log x,\quad \text{and}\label{eqn:mertens1}\\
&\prod_{p\le x}\left(1-\frac{1}{p}\right)\sim \frac{e^{-\gamma}}{\log x},\label{eqn:mertens2}
\end{align}
where $\gamma$ is Euler's constant. It follows from (\ref{eqn:mertens2}) (see \cite[Theorem 328]{HardyWright}) that
\begin{equation}\label{eqn:mertensforphi}
\frac{\varphi(n)}{n}\gg\frac{1}{\log\log n}.
\end{equation}
Next we prove a lemma about pairs of coprime integers in simultaneous arithmetic progressions.
\begin{lemma}\label{lem:arithprogs}
Suppose $m,n,r,s\in\N$ satisfy $(r,s)=1$ and $nr-ms\not= 0$. Then there is a universal constant $\kappa >0$ such that for any $A\in\N$ with
\[A>\kappa\log\log \left(\max (3,|ms-nr|)\right)2^{\omega (ms-nr)},\]
we can find an integer $1\le b\le A$ such that
\[(m+br,n+bs)=1.\]
\end{lemma}
\begin{proof}
Assume without loss of generality, by reversing the roles of the relevant variables if necessary, that $nr-ms>0$. Write $N(A)=N(m,n,r,s,A)$ for the number of integers $1\le b\le A$ such that $(m+br,n+bs)=1.$ By M\"{o}bius inversion we have
\begin{align*}
N(A)=\sum_{b\le A}~\sum_{d|(m+br,n+bs)}\mu (d).
\end{align*}
For each integer $d$ in the inner sum which divides $(m+br,n+bs)$ we can write $m+br=ed$ and $n+bs=fd$, and by reversing the order of summation we obtain
\begin{align}\label{eqn:N(A)1}
N(A)&=\sum_{d\in\N}\mu (d)\sum_{e,f}1,
\end{align}
where the inner sum is over pairs of integers $e$ and $f$ which satisfy the conditions
\begin{align}
1\le e\le (m+Ar)/d,&\quad 1\le f\le (n+As)/d,\label{eqn:efcond1}\\
ed=m\mod r,&\quad fd=n\mod s,\quad\text{ and}\label{eqn:efcond2}\\
\frac{ed-m}{r}&=\frac{fd-n}{s}.\label{eqn:efcond3}
\end{align}
Now to simplify things let us first deal with the case when $(m,r)=(n,s)=1$. For a given $d$ if there are integers $e$ and $f$ satisfying (\ref{eqn:efcond3}) then clearly we must have that $d|nr-ms$.

On the other hand suppose that $d|nr-ms$ and that $e$ is any integer which satisfies the conditions in (\ref{eqn:efcond1}) and (\ref{eqn:efcond2}). Then we claim that there is exactly one choice of $f$ for which (\ref{eqn:efcond1})-(\ref{eqn:efcond3}) hold. To see this write
\begin{equation}\label{eqn:arithprogpf1}
nr-ms=(g-se)d,\quad\text{ with } g\in\Z,
\end{equation}
so that
\[gd=nr+s(ed-m)=0\mod r.\]
Since $(m,r)=1$ and $ed=m\mod r$ we deduce that $(d,r)=1$ and from the equation above we obtain $g=0\mod r.$ Writing $g=fr$ we then see that
\[fd=n+s\left(\frac{ed-m}{r}\right)=n\mod s,\]
and that (\ref{eqn:efcond3}) is satisfied. Furthermore conditions (\ref{eqn:efcond3}), (\ref{eqn:arithprogpf1}), and $1\le e\le (m+Ar)/d$ together imply that $1\le f\le (n+As)/d$. Finally once $e$ and $d$ are chosen there is clearly at most one choice for $f$, so our claim is verified.

Returning to (\ref{eqn:N(A)1}) this shows that when $(m,r)=(n,s)=1$,
\begin{align*}
N(A)&=\sum_{d|nr-ms}\mu (d)\sum_{\substack{e\le (m+Ar)/d \\ e=md^{-1}\mod r}}1 \\
&=\sum_{d|nr-ms}\mu (d)\left(\frac{m+Ar}{dr}+\xi(d)\right),
\end{align*}
for some real constants $\xi(d)$ satisfying $|\xi (d)|\le 1$. This gives us the inequality
\begin{align*}
N(A)&\ge A\sum_{d|nr-ms}\frac{\mu (d)}{d}-\sum_{d|nr-ms}|\mu (d)|\\
&=A\frac{\varphi (nr-ms)}{nr-ms}-2^{\omega (nr-ms)}.
\end{align*}
In the general case if $(m,r)=d_1$ and $(n,s)=d_2$ then since $(r,s)=1$ we have that $(m+br,n+bs)=1$ if and only if
\[\left(\frac{m}{d_1}+b\left(\frac{r}{d_1}\right),\frac{n}{d_2}+b\left(\frac{s}{d_2}\right)\right)=1,\]
and therefore
\begin{align*}
N(m,n,r,s,A)&=N\left(\frac{m}{d_1},\frac{n}{d_2},\frac{r}{d_1},\frac{s}{d_2},A\right)\\
&\ge A\frac{\varphi \left(\frac{nr-ms}{d_1d_2}\right)}{\left(\frac{nr-ms}{d_1d_2}\right)}-2^{\omega \left(\frac{nr-ms}{d_1d_2}\right)}\\
&\ge A\frac{\varphi (nr-ms)}{nr-ms}-2^{\omega (nr-ms)}.
\end{align*}
Here we have used the facts that if $d|M$ then $\omega (M/d)\le \omega (M)$ and
\[\frac{\varphi (M/d)}{(M/d)}=\prod_{p|(M/d)}\left(1-\frac{1}{p}\right)\ge \frac{\varphi (M)}{M}.\]
Our lower bound for $N(A)$, together with (\ref{eqn:mertensforphi}), completes the proof of the lemma.
\end{proof}
We will also use the following elementary result (the proof of which is adapted from an argument in \cite{Erdos}) about prime divisors of integers in short intervals.
\begin{lemma}\label{lem:primedivs}
Let $c>0$ and for $x>1$ set
\begin{equation*}
g_c(x)=2^{\left(c\sqrt{\log x}\right)}\quad \text{and}\quad h_c(x)=\frac{g_c(x)}{\log g_c(x)\log\log g_c(x)}.
\end{equation*}
Then for any $\epsilon >0$ and for all sufficiently large $x$ (depending on $\epsilon$ and $c$), there is at least one integer $N\in [x,x+h_c(x)]$ with
\[\omega (N)\le \frac{(1+\epsilon)\sqrt{\log x}}{c\log 2}.\]
\end{lemma}
\begin{proof}
For $n\in\N$ let
\[\omega_c(n)=\sum_{\substack{p|n\\p>g_c(n)}}1.\]
First of all we have that
\begin{align*}
&\sum_{x\le n\le x+h_c(x)}(\omega (n)-\omega_c(n))\\
&\qquad\qquad\le \sum_{p\le g_c(x+h_c(x))}\sum_{\substack{x\le n\le x+h_c(x)\\p|n}}1\\
&\qquad\qquad\le h_c(x)\sum_{p\le g_c(x+h_c(x))}\frac{1}{p}+\pi(g_c(x+h_c(x))).
\end{align*}
Now by (\ref{eqn:mertens1}), the prime number theorem, and the fact that $g_c(x+h_c(x))\sim g_c(x)$, it follows that there is a number $x_0=x_0(\epsilon,c)$ such that
\begin{align*}
&\sum_{x\le n\le x+h_c(x)}(\omega (n)-\omega_c(n))\\
&\qquad\qquad\le (1+\epsilon)\left(h_c(x)\log\log g_c(x)+\frac{g_c(x)}{\log g_c(x)}\right)\\
&\qquad\qquad =2(1+\epsilon)h_c(x)\log\log g_c(x),
\end{align*}
for all $x\ge x_0$. Therefore for every $x\ge x_0$ there is at least one integer $N\in [x,x+h_c(x)]$ with
\[\omega(N)-\omega_c(N)\le 2(1+\epsilon)\log\log g_c(x),\]
and we have that
\begin{align*}
\omega(N)&\le\omega_c(N)+2(1+\epsilon)\log\log g_c(N)\\
&\le \frac{\log N}{\log g_c(N)}+2(1+\epsilon)\log\log g_c(N)\\
&\le \frac{(1+\epsilon')\sqrt{\log x}}{c\log 2},
\end{align*}
provided that $x$ is sufficiently large.
\end{proof}

\section{Proof of main result}\label{sec:mainproof}
Now we come to the proof of Theorem \ref{thm:main}. Let $\alpha\in\R\setminus\Q$ and let $\gamma\in\R\setminus\{0\}$ (the case when $\gamma =0$ is trivial to verify). There are two cases to consider, depending on whether or not $\gamma\in\alpha\Z+\Z$. The analysis in both cases is fundamentally the same, so we will treat them simultaneously. For each $i\ge 0$ define a pair of integers $m_i$ and $n_i$ as follows. If $\gamma=\alpha\ell+\ell'$ for some $\ell,\ell'\in\Z$ then set
\[m_i=p_i-\ell'~\text{ and }~n_i=q_i+\ell.\]
If $\gamma\not\in\alpha\Z+\Z$ then choose $\ell\in\Z$ so that $\gamma-\ell\in [-\{\alpha\},1-\{\alpha\}),$ and write
\[\gamma-\ell=\sum_{k=0}^\infty b_{k+1}D_k\]
as in Lemma \ref{lem:realcfexp}. Then set
\[m_i=-\ell+\sum_{k=0}^{i-1}b_{k+1}p_k~\text{ and }~n_i=\sum_{k=0}^{i-1}b_{k+1}q_k.\]
Next for each $a,b\ge 0$ define
\[m_i(a,b)=m_i+ap_{i-1}+bp_i~\text{ and }~n_i(a,b)=n_i+aq_{i-1}+bq_i.\]
In the case when $\gamma\in\alpha\Z+\Z$ we have from (\ref{eqn:D_kineq}) that for each $i\ge 1$,
\begin{align*}
|n_i(a,b)\alpha-m_i(a,b)-\gamma|=|(1+b)D_i+aD_{i-1}|\le\frac{1+b}{q_{i+1}}+\frac{a}{q_i}.
\end{align*}
On the other hand in the case when $\gamma\not\in\alpha\Z+\Z$ we have from Lemma \ref{lem:inhomappbnd} that for each $i\ge 4$,
\begin{align*}
|n_i(a,b)\alpha-m_i(a,b)-\gamma|&=\left|\sum_{k=i}^\infty b_{k+1}D_k-aD_{i-1}-bD_i\right|\\
&\le\frac{\max(1,b_{i+1})+b}{q_{i+1}}+\frac{a}{q_i}\\
&\le\frac{b}{q_{i+1}}+\frac{1+a}{q_i},
\end{align*}
using the fact that $b_{i+1}\le a_{i+1}$. In either case we have for $i\ge 4$ that
\begin{equation}\label{eqn:n_im_ibound}
|n_i(a,b)\alpha-m_i(a,b)-\gamma|\le\frac{1+a+b}{q_i}.
\end{equation}
Now consider the quantities
\[N_i(a)=n_i(a,0)p_i-m_i(a,0)q_i.\]
Using (\ref{eqn:cfprop2}) we have that
\begin{align}\label{eqn:N_i(a)}
N_i(a)=n_ip_i-m_iq_i+(-1)^{i+1}a.
\end{align}
We would like to apply Lemma \ref{lem:primedivs} to show that we can find an integer $a$ which is not too large, for which $\omega (N_i(a))$ is also not too large. In order to do this we will verify that $|N_i(0)|\rar\infty$ as $i\rar\infty$. Note that if this were not the case we would still be able to complete the proof (in fact with a better bound), however we still use the extra information for our final calculations.

In the case when $\gamma\in\alpha\Z+\Z$ we have by (\ref{eqn:D_kineq}) and (\ref{eqn:cfprop3}) that
\begin{align*}
N_i(0)&=p_i\ell+q_i\ell'\\
&=q_i\left(\frac{p_i}{q_i}\ell+\ell'\right)\\
&=q_i\left(\alpha\ell+\ell'+\frac{(-1)^{i+1}\xi_i\ell}{q_iq_{i+1}}\right)\\
&=q_i\gamma+\frac{(-1)^{i+1}\xi_i\ell }{q_{i+1}},
\end{align*}
for some constant $0<\xi_i\le 1$. Since $\gamma\not=0$ it is clear in this case that $|N_i(0)|\sim q_i|\gamma|\rar\infty$ as $i\rar\infty$.

In the case when $\gamma\not\in\alpha\Z+\Z$ we have that
\begin{align*}
N_i(0)&=n_iq_i\left(\frac{p_i}{q_i}-\frac{m_i}{n_i}\right)\\
&=n_iq_i\left(\frac{p_i}{q_i}-\alpha+\frac{1}{n_i}(n_i\alpha-m_i)\right)\\
&=n_iq_i\left(\frac{p_i}{q_i}-\alpha+\frac{1}{n_i}\left(\gamma-\sum_{k=i}^\infty b_{k+1}D_k\right)\right).
\end{align*}
Thus by (\ref{eqn:cfprop3}) and Lemmas \ref{lem:inhomappbnd} and \ref{lem:D_ksumsgn} we obtain for $i\ge 4$ that
\begin{align*}
N_i(0)&=q_i\gamma+\frac{(-1)^{i+1}\xi_{i,1}n_i}{q_{i+1}}+\frac{(-1)^{i+1}\xi_{i,2}\max (1,b_{i+1})q_i}{q_{i+1}},
\end{align*}
for some constants $0<\xi_{i,1}\le 1$ and $0<\xi_{i,2}\le 3$. Finally by the uniqueness of the expansion in Lemma \ref{lem:intcfexp} we have that $n_i<q_i$ and we conclude that
\[N_i(0)=q_i\gamma+(-1)^{i+1}\xi_i\]
for some $0<\xi_i\le 4$. As before this shows that $|N_i(0)|\sim q_i|\gamma|\rar\infty$ as $i\rar\infty$.

Now choose $c>0$ and $\epsilon >0$. If $i_0=i_0(\epsilon,c)$ is chosen large enough then it follows from (\ref{eqn:N_i(a)}) and Lemma \ref{lem:primedivs} that for all $i\ge i_0$, there is an integer $1\le a\le h_c(|N_i(0)|)$ with
\begin{equation}\label{eqn:a.choice}
\omega (N_i(a))\le \frac{(1+\epsilon)\sqrt{\log |N_i(0)|}}{c\log 2}.
\end{equation}
There are a couple minor technical points here, namely that $N_i(0)$ could be negative and that it $N_i(a)<N_i(0)$ for half of the values of $i$. However these don't interfere significantly with the proof, only possibly with the choice of $i_0$ above.

Supposing that $1\le a\le h_c(N_i(0))$ is chosen so that (\ref{eqn:a.choice}) is satisfied, we then apply Lemma \ref{lem:arithprogs} with
\[m=m_i(a,0),~n=n_i(a,0),~r=p_i, ~\text{ and }~ s=q_i.\]
We have that
\begin{align*}
&\log\log \left(\max (3,|ms-nr|)\right)2^{\omega (ms-nr)}\\
&\qquad\qquad =\log\log \left(\max (3,|N_i(a)|)\right)2^{\omega (N_i(a))}\\
&\qquad\qquad =o \left(g_{c'}(q_i)\right)~\text{ as }~i\rar\infty ,
\end{align*}
for any $c'>(1+\epsilon)/c\log 2$. Therefore by the lemma, for all $i$ sufficiently large we can find an integer $1\le b\le g_{c'}(q_i)$ with $(m_i(a,b),n_i(a,b))=1$. Then by (\ref{eqn:n_im_ibound}) we have that
\begin{equation*}
|n_i(a,b)\alpha-m_i(a,b)-\gamma|\le \frac{1+h_c(|N_i(0)|)+g_{c'}(q_i)}{q_i}.
\end{equation*}
Now notice that in the above analysis we can always find a suitable $\epsilon>0$, as long as $c,c'>0$ are chosen so that $cc'>\frac{1}{\log 2}$. Therefore by relabeling we may assume that $c=c'>1/\sqrt{\log 2}$ and that $\epsilon>0$ has been chosen so that $c^2<(1+\epsilon)/\log 2$. Also note that we can find a constant $\rho>0,$ depending only on $\gamma$, such that
\[g_c(|N_i (0)|)\le \rho g_c(q_i) \text{ for all } 0\le c\le 1.\]
Putting all of these observations together, we conclude that for any $1/\sqrt{\log 2}<c\le 1$ there is an integer $i_0=i_0(c)$ such that for all $i\ge i_0,$ there are integers $0\le a,b\le g_c(q_i)$ with $(m_i(a,b),n_i(a,b))=1$ and
\[|n_i(a,b)\alpha-m_i(a,b)-\gamma|\le \frac{3(1+\rho)g_{c}(q_i)}{q_i}.\]
For such a choice we also have that $g_c(q_i)\le \rho'g_c(n_i(a,b))$ and that
\[n_i(a,b)\le \rho'q_ig_c(q_i)\le \rho'q_ig_c(n_i(a,b)),\]
where again $\rho'>0$ is some constant that only depends on $\gamma$ (in the case when $\gamma\not\in(\alpha\Z+\Z)$ we can take $\rho'=1$). Substituting back into our inequality above gives
\[|n_i(a,b)\alpha-m_i(a,b)-\gamma|\le \frac{3\rho'(1+\rho)g_{2c}(n_i(a,b))}{n_i(a,b)}.\]
Since this holds for all $1/\sqrt{\log 2}<c\le 1$ the constant $3\rho'(1+\rho)$ can be ignored for large $i$ (i.e. the inequality is always true for a smaller value of $c$ in this interval but possibly with a larger value of $i_0$), and we therefore obtain the statement of the theorem.

\end{document}